\documentclass[11pt,a4paper]{amsart}

\usepackage{latexsym}

\usepackage[dvips]{graphics}
\usepackage{amssymb}
\usepackage{amsthm}
\usepackage{amsmath}
\usepackage{amstext}
\usepackage{amscd}
\usepackage{mathrsfs}
\usepackage{hyperref}
\usepackage{color}
\usepackage{enumitem}

\usepackage{ifmtarg}
\makeatletter
  \newcommand{\TODO}[1]{\@ifmtarg{#1}{\emph{\textcolor{red}{\textbf{TODO}}}~}{\textcolor{red}{ \emph{\textbf{TODO:}~#1~}}}}
  \newcommand{\FORUS}[1]{\@ifmtarg{#1}{\emph{\textcolor{blue}
  {\textbf{TODO}}}~}{\textcolor{blue}{ \emph{\textbf{FOR US:}~#1~}}}}
\makeatother

\numberwithin{equation}{section}

\theoremstyle{plain}
\newtheorem{thm}{Theorem}[section]
\newtheorem{lem}[thm]{Lemma}

\newtheorem{theorem*}{Theorem}[]

\newtheorem{lemma}[thm]{Lemma}
\newtheorem{prop}[thm]{Proposition}
\newtheorem{cor}[thm]{Corollary}

\theoremstyle{definition}
\newtheorem{defi}[thm]{Definition}

\newtheorem{rmk}[thm]{Remark}

\theoremstyle{definition}
\newtheorem{defn}[thm]{Definition}

\theoremstyle{remark}
\newtheorem{rem}[thm]{Remark}

\newtheorem{question}{Question}

\newcommand{\C}{\mathbb{C}}
\newcommand{\D}{\mathbb{D}}
\newcommand{\N}{\mathbb{N}}

\newcommand{\Z}{\mathbb{Z}}

\newcommand{\h}{{h}}

\newcommand{\kk}{{k}}

\newcommand{\DD}{\Delta}

\newcommand{\inv}{^{-1}}

\DeclareMathOperator{\id}{id}

\newcommand{\impli}{\varphi}
\newcommand{\derimpli}{\psi}

\newcommand{\contour}{\mathcal C}
\newcommand{\polar}{\mathcal P}

\newcommand{\point}{P}

\DeclareMathOperator{\sing}{Sing \, }

\DeclareMathOperator{\snu}{mult}
\DeclareMathOperator{\idiscr}{ind}
\DeclareMathOperator{\mult}{mult}

\newcommand{\my}{\mathfrak m}

\usepackage{enumitem}
\setlist[itemize]{labelindent=.6em, itemindent=1em, leftmargin=!, label=\textbullet}
\usepackage{ifmtarg}

\newcommand{\snt}{$\nu$-transverse}
\newcommand{\unt}{$\nu$-transverse }
\newcommand{\multseq}{multiplicity sequence }

\newcommand{\PW}{\mathcal {PW}}

\newcommand{\dimtype}{\operatorname {d.t. }}

\title{Zariski equisingularity of surface singularities in $\C^3$ by a local invariant}

\author{Adam Parusi\'nski and Lauren\c tiu P\u aunescu}
\address {Universit\'e C\^ote d'Azur,  CNRS,  LJAD, UMR 7351, 06108 Nice, France}
\email{adam.parusinski@univ-cotedazur.fr}

\address{School of Mathematics and Statistics, The University of Sydney,
  Sydney, NSW, 2006, Australia }%
\email{laurentiu.paunescu@sydney.edu.au}%

\keywords {
Zariski equisingularity, surface singularities, Teissier's numbers
. }
\subjclass[2010]{
32Sxx,	
14B05. 
32S25, 
}

\begin{document}
\begin{abstract}
We associate to every analytic surface singularity $(V,0)$ in $(\C^3,0)$, not necessarily isolated, an invariant 
$\snu^* (V)$ and show that an analytic family of such singularities $(V_t,0)$, $t\in (\C^l,0)$, is generically Zariski equisingular if and only if $\snu^* (V_t)$ is constant.  The invariant, that we call the \multseq of $V$,  takes into account the multiplicities of
the successive discriminants of $V$ by generic corank one projections.
\end{abstract}
\maketitle


In a seminal Astérisque paper \cite{teissierasterisque73} Bernard Teissier defined for every isolated complex analytic hypersurface singularity $(X_0,0) \subset (\C^{n+1},0)$,  $X_0=f\inv (0)$, a sequence of positive integers 
$$\mu^*(X_0) = (\mu^{n+1}(X_0),\mu^{n}(X_0), \ldots, \mu^1(X_0),\mu^0(X_0)),$$
now called Teissier's numbers. 
Geometrically, $\mu^{i}(X_0)$ is the Milnor number of the section of $X_0$ by a generic linear subspace of dimension $i$, $\mu^1(X_0) = \mult_0(X_0)-1$ and $\mu^0(X_0))=1$.   
Teissier showed in \cite{teissierasterisque73} that if $\mu^*(X_t)$ 
is constant in a family $X \subset \C^{n+1} \times \D $ of isolated singularities parameterised
by $\D=\{t\in \C;  |t|<1\}$, then the pair $(X\setminus (\{0\}\times \D ), \{0\}\times \D )$ satisfies Whitney's Conditions. The converse implication was showed by Briançon and Speder in  \cite{brianconspeder76}.

In \cite[Definition 3] {zariski71open} O. Zariski proposed the notion of  algebro-geometric equisingularity of a hypersurface $X$ along a smooth subspace $S$, now called Zariski equisingularity. In \cite{Zariski1979} Zariski developed this general theory of equisingularity for algebroid and  algebraic hypersurfaces over an algebraically closed field of characteristic zero $\kk$.  This equisingularity notion is defined recursively by taking  the discriminants of some successive co-rank 1 projections. 
Zariski equisingularity of families of isolated surface singularities $(V_t,0)\subset (\C^3,0)$, $t\in T=(\C^l,0)$, was studied by Briançon and Speder  in 
\cite{brianconspederthese} and by Briançon and Henry  in \cite{brianconhenry80}. It was showed in \cite[Théorème 3.5.3]{brianconspederthese} that such a family is Zariski equisingular for a generic linear projection $(x,y,z) \to (x,y)$ if and only if the following numbers are constant
\begin{enumerate}[label=(\roman*)]
  \item 
Teissier's numbers $
\mu^*(V_t) = (\mu^3(V_t),\mu^2(V_t),\mu^1(V_t),\mu^0(V_t)). 
$
\item 
The number of ordinary double points $k(V_t)$ (double-plis évanescents) and the number of cusps 
$\varphi (V_t)$ (fronces évanescentes) of a generic linear projection of the generic fibre of a miniversal deformation of $V_t$.  
\end{enumerate}

It is showed in \cite[Proposition 2. 10]{teissierasterisque73} that $\mu^*(V_t)$ are analytic invariants, that is Teissier's numbers do not depend of a choice of a local system of coordinates.  For $k(V_t)$ and $\varphi (V_t)$ the invariance by an arbitrary local change of coordinates is showed in \cite{brianconhenry80}. Then it is showed in \cite[Théorème 4.1]{brianconhenry80} that a family of isolated surface singularities $(V_t,0)\subset (\C^3,0)$ is generically Zariski equisingular if and only if the numbers $\mu^*(V_t), k(V_t), \varphi (V_t)$ are constant.  In the proof Briançon and Henry give algebraic formulas for $ k(V_t)$ and $ \varphi (V_t)$.  One of their crucial arguments is  \cite[Lemme 3.7]{brianconhenry80}, see also \cite[Lemme clé 1.2.2]{teissierRabida82} for a more general result.  This lemma says  that for a surface $S$ with an isolated singularity and the critical locus $C=C_{\pi}$ of a generic linear projection $\pi$ restricted to $S$, $\pi(C)$ is  equisingular to the generic linear  projection of $C$ ($\pi$ is not generic with respect to $C$ because $C$ was defined using $\pi$).  This lemma, and its proof, also plays a crucial role in the construction of a parametrisation of polar wedges of  Neumann and Pichon  \cite{BNP2014}, \cite{NPpreprint}, and of \cite{PP24}.  

\begin{rmk}
By Speder \cite{speder75}, in any dimension, Zariski equisingularity (for generic projections) implies Whitney's Condition. But in general, except for the case of families of plane curve singularities, it gives stronger equisingularity than Whitney's Condition
 as showed by an example of Briançon-Speder \cite{brianconspeder75b}.  
\end{rmk}

For a non-isolated singularity $(V,0)\subset (\C^3,0)$, neither Milnor's number $\mu^3(V)$, nor $k(V)$ and $\varphi (V)$ can be defined.
  We replace them in Section \ref{sec:localinvariant} by a new invariant $\snu^* (V, 0)$ that we call the \multseq of $V$.  This invariant is defined by the multiplicities of discriminant loci of successive corank 1 generic projections of $V$. As we show in Proposition \ref{prop:multindep} the \multseq  is an analytic invariant and it is upper semicontinuous in families by Proposition \ref{prop:multsemicont}.  For an isolated singularity $V$ it can be expressed in terms of $\mu^*(V), k(V), \varphi (V)$, Proposition \ref{prop:multseqforisolated}.

The following main result of this note is proven in Section \ref{sec:proofofmaintheorem}. 

\begin{thm}\label{thm:mainequivalence2}
Consider the germ of an analytic family of surface singularities in $\C^3$
\begin{align}\label{eq:famwithparameter}
f(x,y,z, t)= f_t(x,y,z) \in \C \{x,y,z,t\},
\end{align} 
where $t\in (\C^l,0)$ is considered as a parameter. 
 Then the following conditions are equivalent:
\begin{enumerate}[label=(\roman*)]
  \item 
After a generic linear change of coordinates $x,y,z$ the family $V_t = f_t\inv (0)$ is  \unt Zariski equisingular.  
  \item 
After a generic change of coordinates $x,y,z$ the family $V_t = f_t\inv (0)$ is   Zariski equisingular.  
\item
The \multseq $\snu^* (V_t, 0)$ is constant.  
\end{enumerate}
\end{thm}

The notion of  \unt Zariski equisingular families was introduced in \cite{PP24}.  We recall it in subsections \ref{ssec:newnut} and \ref{ssec:newZEfamiliesdim2}.  

\begin{cor}\label{cor:indep}
If a family $V_t = f_t\inv (0)$ is  Zariski equisingular after a generic linear change of one system of local coordinates $x,y,z$ in $(\C^3,0)$, then for any system of local coordinates in $(\C^3,0)$ it is \unt Zariski equisingular after a generic linear change of coordinates.      
\end{cor}

\subsection*{Acknowledgements}
{The first author is grateful for the support and hospitality of the Sydney Mathematical Research Institute (SMRI). 
Both authors  acknowledge support from the Project ‘Singularities and Applications’-CF
132/31.07.2023 funded by the European Union-NextGenerationEU-through Romania’s National
Recovery and Resilience Plan.
}


\section{Zariski equisingularity of surface singularities in $\C^3$}
\label{sec:previouspaper}

In this section we recall the basic notions and results of \cite{PP24} on  Zariski equisingularity of families 
of surface singularities in $\C^3$. In \cite{PP24} these results are stated for algebroid singularities, that is the ones defined by formal power series over $\kk$, as in the paper of Zariski \cite{Zariski1979} .    
For the field of complex numbers they can be stated as well for singularities defined by convergent power series, that is for complex analytic singularities.  
It is clear from the definition that an analytic family of complex analytic singularities is Zariski equisingular if and only if it is Zariski equisingular as a family of algebroid singularities, see also \cite[section 1.3]{PP24}.  
For more on Zariski equisingularity including a historical account see \cite{Lip97} and for a recent survey see \cite{MR4367438}. 
\smallskip

Fix a local system of coordinates $x,y,z$ at the origin in $\C^3$ and let  
$f(x,y,z, t)= f_t(x,y,z) \in \C\{x,y,z,t\}$, $f_t(0)=0$,  
 where $t\in (\C^l,0)$ is considered as a parameter. Such an $f$ induces a family surface singularities $V_t=f_t\inv  (0)$ in $\C^3$.  It is often convenient to assume $f$ reduced, or replace it by its reduced form $f_{red}$ if this is not the case, but it is not necessary. 

We always assume that $f$ is regular in $z$, that is $f_0(0,0,z) \not \equiv 0$.  
Then $f$ is equivalent to a Weierstrass polynomial in $z$ that  it is of the form  
  \begin{align}\label{eq:family}
f_t(x,y,z)= u (x,y,z, t) (z^{d}+ \sum_{i=1}^{d} a_{i} (x,y,t)), \quad u(0)\ne 0 ,   
\end{align} 
with the coefficients satisfying $a_{i}(0)= 0$.  We denote by $D_f (x,y, t)$ the discriminant of $f$, 
more precisely the discriminants of the Weierstrass polynomial associated to $f$, and by $D^i_f (x,y, t) $ the $i$-th generalized discriminant of $f$, see the Appendix.  If $f$ is reduced then we only need the standard discriminant $D_f$.  

In general, let $g\in \C \{x_1, \ldots , x_n\}$ be not necessarily reduced and suppose that 
  $g$ is regular in the variable $x_n$.  We denote by $D_g\in \C \{x_1, \ldots , x_{n-1}\} $ its discriminant and by $D^i_g \in \C \{x_1, \ldots , x_{n-1}\} $ its generalized (or higher order) discriminants.  Then by definition $D_g=D^1_g$.  By  $\idiscr (g) $ we denote the smallest $i$ for which $D^i_g$ is not identically equal to zero.  Then,  up to a universal constant, this first non identically zero $D^i_g$ is equal to the standard discriminant of $g_{red}$.  We refer the reader to the Appendix  for more details on this notion.  
 
\begin{defi}\label{def:newZAinfamilies}
Let $f(x,y,z, t)= f_t(x,y,z) \in \C\{x,y,z,t\}$ be as above.  We say that the family of set germs $V_t = f\inv (0)$ is \emph{Zariski equisingular with respect to the parameter $t$ and the system of coordinates $x,y,z$} if 
\begin{enumerate}
\item 
$D^{j_0}_f (x,y, t) $ is regular in $y$, where $j_0 = \idiscr (f) $,
\item
Either $D^{j_0}_f (0) \ne  0 $ or 
the first non identically equal to zero generalized discriminant of $D^{j_0}_f $ 
is equimultiple along $\{(0,t)\in \C\times \C^l\}$, that is 
$$\mult_{(0,t)} (D_{(D^{j_0}_f)_{red}}) = \mult_{(0,t)} (D^{i_0}_{(D^{j_0}_f)}) ,$$ where $i_0 = \idiscr (D^{j_0}_f) $,  is independent of $t$.  
\end{enumerate}
\end{defi}

We may summarize these two conditions by saying that the discriminant locus of $f_{red}$, denoted $\Delta_f$ and defined as the zero set of $D_{f_{red}}$,  is a Zariski equisingular family of plane curve singularities with respect to the parameter $t$ and the projection $(x,y) \to x$. 

We say that the family $V_t$ is \emph{generically linearly Zariski equisingular (with respect to the parameter $t$ and the system of coordinates $x,y,z$)} if it is Zariski equisingular after a generic linear change of coordinates $x,y,z$, that is, for a change from a Zariski open and dense subset of linear changes of coordinates.  The notion of \emph{generic Zariski equisingularity} 
is more complicated because the space of all possible local changes of  system of coordinates is infinitely dimensional.  We refer the reader to \cite{PP24} for a discussion of this notion.  A result of \cite{PP24} says that for a family of surface singularities  in $\C^3$ the notion of generic and generic linear Zariski equisingularity coincide.  This is still an open question for hypersurface singularities  in $\C^n$ for $n\ge 4$.

For a family of plane curve singularities, it is known that we can take as a generic system of coordinates $x,y$ any system for which the kernel of the projection $(x,y)\to x$, i.e. the $y$-axis, is transverse to the curve, see \cite{zariski65-S2}.  

 In general, if a hypersurface $X$ is the zero set of $g\in \C \{x_1, \ldots , x_n\}$, then the $x_n$-axis is transverse to $X$ if it is not included in the tangent cone $C_0(X)$ to $X$ at the origin.  In our case 
we say that the projection $(x,y,z,t) \to  (x,y,t)$ is  \emph{transverse} to $V$
if its kernel is not included in the tangent cone $C_0(V)$. 
This is equivalent to saying the multiplicity of $V$ at $0$ being  equal 
to $d$ if $f$ is given by \eqref{eq:family}.

\begin{rmk}\label{rk:newtransversality and equimultiplicity}
One important feature of the Zariski equisingularity is that it implies equimultiplicity, see \cite{zariski75} 
for the general case and \cite[page 531]{zariski65-S1} for families of plane curve singularities.    
The equimultiplicity of $V_t$, as family of reduced hypersurfaces, 
implies the normal pseudo-flatness of $V$ along $\{0\}\times (\C^l,0)$,  
intuitively the continuity of the tangent cone, see e.g. \cite{Hironaka69}
for the notion of normal pseudo-flatness.  Therefore, in Zariski equisingular families, 
if $V_0$ is transverse to the kernel of the projection so are $V_t$ for $t$ non-zero.
\end{rmk}

We say that the system of coordinates $x,y,z,$ for the surface singularity 
$V=f\inv(0)$ is \emph{transverse}  if the $z$-axis is transverse to $V$ at the origin in $\C^3$ and the $y$-axis is transverse to the 
zero set of the discriminant $D_f$ at the origin in $\C^2$. 
However, for a projection to be generic in the sense of Zariski equisingularity, it is not sufficient to assume that it is only transverse. In \cite{luengo85} Luengo gave an example of a family of surface 
singularities in $\C^3$ that is Zariski equisingular for one transverse 
projection but not for generic linear ones. 
In \cite{PP24} we introduced a new notion of the nested uniformly transverse system of coordinates and showed that it is sufficiently generic for the purpose of Zarski equisingularity.


\subsection{Nested uniformly transverse system of coordinates}\label{ssec:newnut}
Given  an analytic hypersurface $V = f^{-1} (0) \subset (\C^3,0 )$ defined by a reduced convergent power series $f\in \C\{x,y,z\}$.  
 For $b\in \C$ we denote by $\pi_b$ the projection 
$\C^3\to \C^ 2$ parallel to $(0,b,1)$, that is $\pi_b(x,y,z) = (x, y-bz)$.  
If $\pi_b$   restricted to $V$ is finite, for instance if its kernel is transverse to $V,$
then we denote by $\Delta_b$ the discriminant locus of such restriction $\pi_b{_{|V}}$.

\begin{defn} \label{defn:newsnt-generic}
We say that a local  system of coordinates $x,y,z$ at $0$
is \emph{nested-uniformly-transverse}, or \emph{\snt} for 
short, for $V$  if 
\begin{enumerate}
\item
The projection $(x,y,z) \to (x,y) $ is transverse to $V$, i.e., the tangent  cone $C_0(V)$  does not contain the $z$-axis. 
  \item
 The projection $(x,y) \to x $ is transverse to $\Delta_0$ at $0$.
 \item 
The family of singularities of plane curves $\Delta_b$ parameterised by 
 $b\in \C$ is equisingular for small $b$.  
\end{enumerate}   
\end{defn}

\begin{lemma}[{\cite[Lemma 4.2]{PP24}}]\label{lem:newgenericlinear==>nu}
Given a local coordinate system  at the origin in $\C^3$.  Then, after performing a generic linear change of coordinates, i.e.  from a Zariski open non-empty subset of linear changes of coordinates, the new system is \snt.  
\end{lemma}


The following two theorems show that for a \unt system of coordinates the projection $\pi_0$ is generic in the sense of Zariski Equisingularity,  with respect to both linear, Theorem \ref{thm:newgenlin},  and nonlinear, Theorem \ref{thm:newgeneric}, changes of coordinates. 

\begin{thm}[{\cite[Theorem 4.4]{PP24}}]\label{thm:newgenlin}
Suppose $x,y,z$ is a \unt system of coordinates for $V= V(f(x,y,z))$. 
 Then the family $V(f_{a,b})$ for 
\begin{align}\label{eq:newlinearfamily}
f_{a,b}(X,Y,Z):= f(X+aZ,Y+bZ,Z), 
\end{align}
is Zariski equisingular with respect to the parameters $(a,b)\in (\C^2,0)$. 
\end{thm}

Consider the  family 
\begin{align} \label{eq:newgenericfamily}
f_\tau (X, Y, Z)= f(\impli_\tau  (X,Y,Z)) ,
\end{align}
where $\impli_\tau $ is a local, not necessarily linear, analytic change of coordinates depending analytically on a parameter $\tau\in (\C^s,0)$, 
i.e. $\impli \in (\C\{ x,y,z,\tau\}) ^3$.

\begin{thm}[{\cite[Theorem 4.4]{PP24}}]\label{thm:newgeneric}
Suppose $x,y,z$ is a \unt system of coordinates for $V= V(f(x,y,z))$.  Consider a local change of coordinates $\impli_\tau  (X,Y,Z) $ that depends analytically on a parameter 
$\tau\in (\C^s,0)$ and satisfies $\impli_\tau  (0)=0$ for all $\tau$, and $D\varphi_\tau  (0)=\id$ for $\tau =0$.  

Then $V(f(\impli_\tau  (X,Y,Z)))$ is Zariski equisingular with respect to the parameter $\tau $.  
\end{thm} 



\subsection{\unt Zariski equisingular families}
\label{ssec:newZEfamiliesdim2}

Consider an analytic family 
\begin{align}\label{eq:newfwithparameter}
f(x,y,z, t)= f_t(x,y,z) \in \C\{x,y,z,t\}, \quad f_t(0)=0,
\end{align} 
with all $f_t$ reduced, where $t\in (\C^l,0)$ is a parameter. It induces a family surface singularities $V_t=f_t\inv  (0)$ in $\C^3$.

\begin{defn} \label{defn:newsnt-genericZE}
We say that the family $V_t$ is \emph{\unt Zariski equisingular} (with respect to a given system of coordinates) if this  system of coordinates 
$x,y,z$ is \unt for $V_0= f_0^{-1} (0)$ and the family of the discriminant loci $\Delta_{b,t} = 
D_f\inv (0)$ of the projection 
$(x,y,z,t) \to (x, y-bz,t)$  parameterized by $(b,t) \in (\C\times \C ^l,0)$ 
is an equisingular family of plane curve singularities.    
\end{defn}

Theorems \ref{thm:newgenlin} and \ref{thm:newgeneric} are special cases of Theorem 
\ref{thm:newgenericwithparameter}. 
Consider an arbitrary  analytic change of local coordinates 
\begin{align}\label{eq:newchange} (x,y,z,t) = \impli (X,Y,Z,T), 
\end{align}
i.e.  $\impli (X,Y,Z,T)  = (\impli_1, \impli_2, \impli_3, \impli_4)$ with $\impli_4 $ 
depending not only on $T$ but also on $X,Y,Z$.

\begin{thm}[{\cite[Theorem 4.6]{PP24}}] \label{thm:newgenericwithparameter}
Suppose that the family $V_t$  is  \unt   Zariski equisingular.  Let \eqref{eq:newchange} be a local change of coordinates that depends analytically on a  parameter $\tau\in (\C^s,0)$ and satisfies:  $\impli_1 (0,T)\equiv \impli_2 (0,T)\equiv \impli_3 (0,T) \equiv 0$ for all $\tau$, and $D\varphi (0)=\id$ for $\tau =0$.  

Then, the family of the zero sets of $F (X, Y, Z, T)= f(\impli (X,Y,Z,T))$ is Zariski equisingular
 with respect to  parameters  $T$ and $\tau$.  
\end{thm} 


Theorem \ref{thm:newgeneric} follows immediately from 
Theorem \ref{thm:newgenericwithparameter}.  To show 
Theorem \ref{thm:newgenlin} it suffices to consider the local change of variables 
$(x,y,z) \to (x-az,y-bz, z)$ parameterized by $\tau =(a,b)$ and apply Theorem \ref{thm:newgenericwithparameter}.

We will also need the following version of \cite[Lemma 5.2]{PP24}.  

\begin{lem}\label{newlem:(4)==>(1)}
If a family of surface singularities $V_t$ is Zariski equisingular in a generic linear system of coordinates then it is  \unt Zariski equisingular.  
\end{lem}

We will sketch its proof because we need a version for Zariski equisingularity of families of singularities and not, as in \cite[Lemma 5.2]{PP24}, for Zariski equisingularity along a subspace.  
Its proof is based on  \cite[Proposition 3.7]{PP24} that is actually stated in \cite{PP24} for 
Zariski equisingularity of families of singularities and not for Zariski equisingularity along a subspace. Thus the above Lemma \ref{newlem:(4)==>(1)} is an easier conclusion of \cite[Proposition 3.7]{PP24} than Lemma 5.2 of \cite{PP24}. 

First, let us recall the statement of \cite[Proposition 3.7]{PP24}.

\begin{prop}[{\cite[Proposition 3.7]{PP24}}]\label{prop:rationalinparameter}
Suppose that $f(x,t,\tau)\in \kk(\tau) [[x,t]]$, $f(0,t)\equiv 0$, is uniformly rational in $\tau$ and regular in $x_{r+1}$, where $x = (x_1, \ldots, x_{r+1})$, $t=(t_1, \ldots , t_l), \tau=(\tau_1, \ldots, \tau_s)$.  Denote by $\pi$  the projection 
$(x_1, \ldots, x_r, x_{r+1}, t)\to (x_1, \ldots, x_r, t)$. 
Then the following conditions are equivalent:
\begin{enumerate}
  \item 
$V= V(f)$, understood as an algebroid  hypersurface over $\kk(\tau)$, 
is Zariski equisingular with respect to the parameter $t \in (\kk^l,0)$ (and the projection $\pi$).  
\item
There is a non-empty Zariski open  subset of $U \subset \kk ^s$ such that 
for every $ \tau \in U$ the family 
$V_{ \tau } =\{f(x,t,\tau) =0\}$ 
is Zariski equisingular with respect to the parameter $t \in (\kk^l,0)$ (and the projection $\pi$).  
\item
There exists $ \tau_0 \in \kk ^s$ such that $V$, understood 
as an algebroid hypersurface of $(\kk^{r+1+l+s},\,  (0,0,\tau_0))$, 
is Zariski equisingular with respect to the local parameters 
$(t,\tau-  \tau_0) \in (\kk^l \times \kk^s,0)$ (and the projection $(x_1, \ldots ,x_r, x_{r+1}, t,\tau)\to (x_1, \ldots ,x_r, t, \tau)$).  
\end{enumerate}
\end{prop}

\begin{defn}[{\cite[Definition 3.5]{PP24}}]
We say that a power series $f\in \kk(\tau) [[x]]$, 
\begin{align}
f(x,\tau) = \sum_\alpha a_\alpha(\tau ) x^\alpha,  
\end{align}
$x=(x_1, \ldots, x_{r+1}) $, 
$\tau=(\tau_1, \ldots, \tau_s)$, \emph {is uniformly rational in $\tau$}
if the  coefficients of $f$ are of the 
form $a_\alpha = b_\alpha/c^{|\alpha|}$, 
where $b_\alpha, c $ are polynomials $b_\alpha, c \in \kk[\tau]$, 
and $c\ne 0$, or, equivalently, $f\in \kk[\tau] [[\frac x c]]$. 
\end{defn}

The notion of uniformly rational of \cite[Section 3.3]{PP24}   
is a version of Rond's Eisenstein's series, see  \cite{rond18, rondeisensteins, ranktheorems}.  
This notion has two convenient properties that we used in the proof of Proposition \ref{prop:rationalinparameter}.  
Firstly, for a power series $f$ uniformly rational in $\tau$,  the 
 specialization $\overline \tau \to \tau\in \kk$ makes sense for $\overline \tau$ from a non-empty Zariski open subset, more precisely for 
 if $c(\overline \tau) \ne 0$.  
Secondly,  it follows from a classical proof of the Weierstrass preparation theorem that the property of being
 uniformly rational in $\tau$ is preserved by taking the associated Weierstrass polynomial, and hence by taking the discriminant of $f$ as well. See \cite[Section 3.3]{PP24} for details.

\begin{proof}[Proof of Lemma \ref{newlem:(4)==>(1)}]
We apply Proposition \ref{prop:rationalinparameter} to the series $f(X,Y+bZ, Z, t)$ whose coefficients are polynomials in $b$,  and hence, trivially, uniformly rational.

Let $V_t= f_t\inv (0)$, where $f$ is as in \eqref{eq:famwithparameter}.  
By assumption for all but finite values of $b\in \C $ the discriminant of 
$f$ with respect to the projection 
$$
\pi_b(x,y,z,t) = (x, y-bz,t) 
$$ 
is Zariski equisingular as a family of plane curve singularities depending on $t$. 
Now the implication $(2) \Longrightarrow (3)$ of Proposition \ref{prop:rationalinparameter} shows that 
$f(X,Y+bZ, Z, t)$ is a Zariski equisingular family parameterized by 
$(t,b)$, for all but finitely many $b$.  This implies that the family is   \unt Zariski equisingular.  
\end{proof}


\section{The \multseq }
\label{sec:localinvariant}

Given  an analytic hypersurface $V = f^{-1} (0) \subset (\C^3,0)$ defined by an  
$f\in \C \{x,y,z\}$ that we assume reduced.  The \multseq of $V$ takes into account the multiplicities of the successive discriminants of $f$.  
Since the discriminant $D_f$ of $f$ is not necessarily reduced we have to first replace $D_f$ by its reduced form $(D_f)_{red}$ or,  what is more convenient, to use the generalized discriminants of $D_f$.  



\begin{defn} \label{defn:untmultiplicitysequence}
Suppose that for a germ $V= f\inv (0)$ with $f$ reduced a local  system of coordinates 
$x,y,z$ at the origin is \unt.  
Then the \emph{\multseq of $(V,0)$}  is defined by 
$$\snu^* (V) = (\mult_0  (V), \mult_{0} (D_f ), i_0,  
\mult_{0} 
 (D^{i_0}_{D_f} )),$$
where $i_0=\idiscr (D_f) $.
\end{defn}

In the above defintion the value of $i_0=\idiscr (D_f) $ is important.  We do not get the same information if we just replace $D^{i_0}_{D_f} $  by $D_{(D_f)_{red}} $. 

Note that the  \multseq is defined with respect to a  \unt system of coordinates and 
a priori may depend on this  system coordinates.  We show that it is in fact independent of such a choice.  
For this we first remark that the \multseq  is constant in \unt Zariski equisingular families.  

\begin{prop}\label{prop:families}
Given an analytic family of surface singularities $V_t=f_t\inv (0)$ in $\C^3$  given by
\eqref{eq:newfwithparameter}.   
If this family is \unt Zariski equisingular then the \multseq $\snu^* (V_t,0)$ 
(with respect to the fixed system of coordinates $x,y,z$)  is independent of $t$.  
\end{prop}

\begin{proof}
It follows from the fact that the multiplicity is constant in Zariski equisingular families, see Remark \ref{rk:newtransversality and equimultiplicity}. 
\end{proof}


\begin{prop}\label{prop:multindep}
The \multseq of $V$ is independent of the choice of local system of coordinates.    
\end{prop}

\begin{proof}
We use Theorem \ref{thm:newgeneric} and Proposition \ref{prop:families}. 

Let $(x,y,z) = \impli (X,Y,Z)$ be an analytic change of local coordinates $(x,y,z) = \impli (X,Y,Z)$.  We put it in a family 
\begin{align}\label{eq:systemdeformation} 
(x,y,z)= \Phi (X,Y,Z,t)= (1-t)(X,Y,Z) + t \impli (X,Y,Z), \quad t\in \C,     
\end{align}
that defines the family   
$F(X,Y,Z,t)=F_t(X,Y,Z)=  f(\Phi (X,Y,Z,t))$ and   $V_t=F_t\inv (0)$.   Note that for this family the system of local coordinates $X,Y,Z$ is fixed and the hypersurface $V_t=F_t\inv (0)$ is changing. 
If this family is \unt equisingular (at least for all $t$ on a continuous curve in $\C$ joining $0$ and $1$) then, by Proposition \ref{prop:families},  \multseq of $V$ is the same in the system given by $t=0$ and the one given by $t=1$.   

Firstly, we remark that the \multseq is preserved by a small linear change of coordinates.  Indeed, it follows from Theorem \ref{thm:newgenlin} that shows in particular that a small linear change of a \unt system of coordinates is \snt.  Then the statement for an arbitrary linear change of coordinates between two \unt system of coordinates follows from the fact that the multiplicity is generically constant in analytic families of singularities.

Secondly, we consider an arbitrary analytic change of local coordinates 
$(x,y,z) = \impli (X,Y,Z)$ with an extra condition $D \impli (0) = \id $.  

\begin{lemma}\label{lem:nutransverse1}
  If we suppose in \eqref{eq:systemdeformation}  that $D \impli (0) = \id $ and that the system of coordinates for $t=0$ is \unt 
  then for any $t_0 \in \C $ the system $X,Y,Z$ is \unt for $V_{t_0}$ 
 and the family $V_{t}$ is \unt Zariski equisingular at $t_0$. 
\end{lemma}

\begin{proof}
  This is a consequence of Theorem \ref{thm:newgeneric}. Indeed, apply this theorem to   $\varphi (X,Y,Z,T, b) = \Phi ( X, Y+bZ,Z,T-t_0 )  $, $T$ and $b$ being parameters.  
 \end{proof}


Finally, for an arbitrary analytic change of local coordinates,
 we first use Lemma \ref{lem:nutransverse1} to join the given system of local coordinates to its linear part, then use the result for a linear change of coordinates.   
\end{proof}


\section{\multseq for isolated singularities}
\label{sec:isolsing}



For isolated singularities the \multseq can be computed in terms of $\mu^*(V_t)$, $k(V_t)$, and $\varphi (V_t)$ and, as we explain below, the constancy of $\mu^*(V_t)$, $k(V_t)$, and $\varphi (V_t)$ is equivalent to the constancy of $\snu^* (V_t, 0)$.  Thus Théorème 3.5.3 of \cite{brianconspederthese} and Théorème 4.1 
of \cite{brianconhenry80} are consequences of Theorem \ref{thm:mainequivalence2}.  

For an isolated surface singularity $(V,0)\subset (\C^3,0)$
the following formulas are given in \cite{brianconspederthese}, see also \cite[Proposition 1.1]{brianconhenry80}, the first one also follows from  \cite[Lemme 5.10]{teissierasterisque73}, 
\begin{enumerate}[label=(\alph*)]
\item 
the multiplicity of the discriminant locus $\Delta_{f}= D_{f}\inv(0)$ is given by 
\begin{align}\label{eq:formula1}
\mult_0 (\Delta_{f}) = \mu^2(V) + \mu^1(V) .
\end{align}
\item 
the Milnor number of $\Delta_{f}$ is given by
\begin{align}\label{eq:formula2}
\mu (\Delta_{f}) = \mu^3(V) - \mu^1(V) + 2 k(V) +3\varphi (V) +1 .
\end{align}
\end{enumerate}

Note that for a generic linear projection $D_{f}$ is reduced (by the isolated singularity assumption) and hence $\idiscr (D_{f}) =1 $.  
Then, as follows from  \cite[Lemme 5.10]{teissierasterisque73},  
\begin{align}\label{eq:formula3}
\mult_0 (D_{D_{f}}) = \mu (D_{f}) + \mult_0 (D_{f}) -1 .
\end{align}

We have the following result.

\begin{prop}\label{prop:multseqforisolated}
For an isolated surface singularity $(V,0)\subset (\C^3,0)$ the \multseq equals 
\begin{align}\label{eq:multseqforisolated}
\snu^* (V)  = (\mu^1 (V) + \mu^0 (V), \mu^2 (V) + \mu^1 (V), 1, \mu^3 (V) + \mu^2 (V) + 2 k(V) +3\varphi (V) ).
\end{align}
In a family of isolated surface singularities $(V_t,0)\subset (\C^3,0)$ the constancy of the 
\multseq is equivalent to the constancy of $\mu^*(V_t)$, $k(V_t)$, and $\varphi (V_t)$.  
\end{prop}

\begin{proof}
The formula \eqref{eq:multseqforisolated} follows from the formulas \eqref{eq:formula1}, \eqref{eq:formula2}, \eqref{eq:formula3}.  Therefore the con{}stancy of $\mu^*(V_t)$, $k(V_t)$, and $\varphi (V_t)$ implies the constancy of the multiplicity sequence.  For the converse we note that all these invariants are upper semicontinuous in families.  For Teissier's numbers it follows from \cite{teissierasterisque73}, for $k(V_t)$ and $\varphi (V_t)$ it is shown in \cite{brianconhenry80}, and for the \multseq it is given  by Proposition \ref{prop:multsemicont}.  
\end{proof}

\section{Proof of Main Theorem}
\label{sec:proofofmaintheorem}

\begin{proof}[Proof of Theorem \ref{thm:mainequivalence2}]
The condition (i) clearly implies (ii) and the converse follows from Lemma \ref{newlem:(4)==>(1)}.  

By Proposition \ref{prop:families}, (i) implies (iii).  We will show the converse. Thus suppose (iii).  After a generic linear change of the system of coordinates we may suppose that the local system of coordinates $x,y,z$ is \unt for $t=0$.  By the assumption (iii),  $\mult_0 {f_t}$ is independent of $t$ and hence, see Remark \ref{rk:newtransversality and equimultiplicity},  the kernel of the projection $\pi_b$ is transverse to $V_t$ for $t$ and $b$ sufficiently small. 

We will first show that the \multseq $\snu^* (V_{t,b})$ of $V_t$ with respect to the system of coordinates $(x, y-bz, z)$ is constant for $b$ and $t$ sufficiently small.  

Let $D_f (x,y,t,b)$ denote the discriminant of $f_t$ with respect to the projection $\pi_b$.   Now we argue for $(D_f)_{t,b}$ as before for $f_t$.  By the semicontinuity of multiplicity, and the fact that the multiplicity $\mult_0 (D_f)_{t,b}$ at $b=t=0$ equals to the multiplicity at $t,b$ small and generic, this multiplicity is independent of $t,b$  and the projection $(x,y)\to x$ is transverse for the zero sets of $(D_f)_{t,b}$ provided $t$ and $b$ are sufficiently small.  

The set of parameters $t,b$ for which $D^i_{D_f} \equiv 0$ as a function of  $x,y$ is closed for all $i$. Therefore, again by assumption (iii),  $D^i_{D_f} \equiv 0$ and $D^{i_0}_{D_f}\not \equiv 0$ for  $i<i_0=\idiscr ({D_{f_0}}) $ and all $t$ and $b$ sufficiently small.    We conclude this part of the proof again by the semicontinuity of the multiplicity, this time  of $D^{i_0}_{D_f}$.    

Now we show that the family of discriminant loci $\Delta_{t,b}$ of the projection $\pi_b$ of $V_t$ is an equisingular family of plane curves. Since the projection $(x,y) \to x$ is transverse to $\Delta_{0,0}$ it suffices to show that the family that the discriminants of projection onto $x$-axis of $\Delta_{t,b}$ is equisingular, that is equimultiple as a family of hypersurfaces in 
$\C$.  This follows form the constancy of $\snu^* (V_{t,b})$.    
\end{proof}

The proof of the semicontinuity  of the \multseq with respect to the lexicographic order is similar.

\begin{prop}\label{prop:multsemicont}
The \multseq is upper semicontinuous in families.    
\end{prop}

\begin{proof}
We consider the family $V_t= f_t\inv (0)$ given by  \eqref{eq:newfwithparameter} and  show that $\snu^* (V_t, 0) \le \snu^* (V_0, 0)$ for $t$ close to $0$ and that 
$\{t;  \snu^* (V_t, 0) = \snu^* (V_0, 0)\}$ is an analytic germ.  
We suppose that the local system of coordinates $x,y,z$ is \unt for $t=0$.  

The multiplicity $\mult_0 V_t$ is clearly semicontinuous. Thus we may restrict the parameter space to any irreducible component of the set $\{t; \mult_0 V_t = \mult_0 V_0 \}$.  Denote such a component by $T$.  (Using resolution of singularities we may assume that $T$ is nonsingular but this is not necessary.)  Next we consider $D_{f_t}(x,y,b)$ for $t\in T$ and $b$ close to the origin.  Again we may restrict to an irreducible component of $\{t\in T; \mult_0 D_{f_t} = \mult_0 D_{f_0} \}$.  We keep the notation $T$ for such a component.   
The set of parameters $t,b$ for which $D^i_{D_f} \equiv 0$ is closed for all $i$.  Thus again, by restricting to an irreducible component of an analytic subset ot $T$, we may suppose that $D^{i_0}_{D_{f_t}} \not \equiv 0$ for all $t\in T$, where $i_0=\idiscr (D_{f_0}) $.  
Finally the $\{t\in T; \mult_0 D^{i_0}_{D_{f_t}} = \mult_0 D^{i_0}_{D_{f_0}} \}$ is closed, that completes the proof.  
\end{proof}




 \section*{Appendix. Generalized discriminants}

We recall the notion of generalized (or subdiscriminants) discriminants, see Appendix IV of \cite{whitneybook}, \cite{BPRbook}, \cite{roy2006} or \cite{MR4367438}. 
Let $\kk$ be an arbitrary field of characteristic zero and let 
\begin{align}\label{type}
F(Z) = Z^d + \sum_{j=1}^ d a_iZ^{d-i} = \prod_{j=1}^d (Z-\xi_i)\in \kk [Z], 
\end{align}
be a polynomial with coefficients $a_i\in \kk$.  Let $\xi_1$,\ldots, $\xi_d$ be the roots of $F$ in the algebraic closure $\overline \kk$ of $\kk$.  Then \emph{the generalized discriminants (or subdiscriminants)} 
can be defined in terms of the roots 
$$
D_{d-j+1} =  \sum_{r_1 <\cdots < r_{j}} \,  \prod_{k< l;\, k,l \in \{r_1, \ldots ,r_{j}\}}  (\xi_k-\xi_l)^2. 
$$
Since the $D_{d-j+1}$ are symmetric polynomials in the $\xi_i$ they are polynomials in the coefficients $a_i$ and one can even show that they are polynomials with integer coefficients, that is belong to $\Z[a_1, \ldots, a_d]$.  
In particular $D_1$ is the classical discriminant of $F$, and $F$ admits exactly $k$ distinct roots in $\overline \kk$ if and only if $D_{1} = \cdots = D_{d-k}=0$ and 
$D_{d-k+1} \ne 0$.

If $F$ is not reduced, that is if it is has multiple roots, then the generalized discriminants can replace the (classical) discriminant of $F_{red}$.   

\begin{lemma}\label{lem:twodiscr} 
Suppose $F$ has exactly $s$ distinct roots in 
$\overline {\kk}$ of multiplicities $\mathbf m=(m_1, ..., m_s)$.  Then there is a positive constant $C= C_{s,\mathbf m}$, depending only on 
$\mathbf m=(m_1, ..., m_s)$, such that  the generalized discriminant $D_{s}$ of $F$ and 
the standard discriminant $D_{F_{red}}$ of $F_{red}$ are related by the formula $D_{d-s+1,F}= C D_{F_{red}}$. 
\end{lemma}



\bibliographystyle{siam}
\bibliography{ZE}

\end{document}